\documentclass[12pt,leqno]{article}
\usepackage{amsmath,amsthm,amsfonts,amssymb}

\makeatletter
\def\@seccntformat#1{\csname the#1\endcsname.\quad}
\renewcommand\section{\@startsection {section}{1}{\z@}%
                     {-3.5ex \@plus -1ex \@minus -.2ex}%
                     {2.3ex \@plus.2ex}%
                     {\normalfont\large\bf\centering}}
\makeatother

\numberwithin{equation}{section}

\theoremstyle{plain}
\newtheorem{theorem}[equation]{Theorem}
\newtheorem{lemma}[equation]{Lemma}

\newtheorem{proposition}[equation]{Proposition}

\theoremstyle{definition}
\newtheorem{definition}[equation]{Definition}

\theoremstyle{remark}
\newtheorem{remark}[equation]{Remark}
\newtheorem*{remark*}{Remark}

\newcommand{\kyosu}{\sqrt{-1}\,}
\newcommand{\XH}{\widehat{X}}
\newcommand{\NH}{\widehat{\nabla}}
\newcommand{\pa}{\partial}
\newcommand{\ep}{\varepsilon}
\newcommand{\al}{\alpha}
\newcommand{\be}{\beta}
\newcommand{\ov}{\overline}

\renewcommand{\Re}{\mathrm{Re}}
\renewcommand{\Im}{\mathrm{Im}}

\allowdisplaybreaks

\title{A concrete approach to diagonal short time asymptotics of 
heat kernels associated with sub-Laplacian on CR manifolds}
\author{ Hiroki Kondo 
 \thanks{\ E-mail : h-kondo@math.kyushu-u.ac.jp}
\\ {\small\textit{Graduate School of Mathematics, Kyushu University}}
}
\date{}
\begin{document}
\maketitle

\begin{abstract}
A diffusion process associated with the real sub-Laplacian $\Delta_b$, 
the real part of the complex Kohn-Spencer Laplacian $\square_b$, 
on a strictly pseudoconvex CR manifold is constructed in \cite{kondo-taniguchi}. 
In this paper, we investigate diagonal short time asymptotics of 
the heat kernel corresponding to the diffusion process by using 
Watanabe's asymptotic expansion and give a representation for the asymptotic expansion
of heat kernels which shows a relationship to the geometric structure. 
\medskip

\end{abstract}

\section*{Introduction}
A diffusion process associated with the real sub-Laplacian $\Delta_b$ on a 
strictly pseudoconvex CR manifold $M$ is constructed and 
the existence of $C^{\infty}$ heat kernel 
$p(t,x,y)$, $t\in[0,\infty)$, $x,y\in M$ for $\Delta_b$ is studied in 
\cite{kondo-taniguchi}. 
As a question arising naturally from this result, 
we consider diagonal short asymptotics of the heat kernel, 
i.e. asymptotic behavior of $p(t,x,x)$ as $t$ tends to zero. 

In \cite{kondo-taniguchi}, the diffusion process on $M$ is constructed 
via the Eells-Elworthy-Malliavin method. 
More precisely, a stochastic differential equation (SDE in abbreviation) on
a complex unitary frame bundle is considered and 
its unique strong solution is projected onto $M$, then 
the resulting process on $M$ is the desired diffusion process generated by $-\Delta_b/2$. 

Since the diffusion process is obtained as a strong solution of an SDE, 
we can utilize the asymptotic expansion theory due to 
Watanabe (see \cite{i-w}) to investigate asymptotic behavior of 
the heat kernel associated with the diffusion process. 
To be more precise, if the solution $U^{\ep}_t$, $t\in[0,\infty)$ of the SDE
which is obtained by putting a parameter $\ep>0$ into the 
SDE defining the diffusion process has a suitable expansion in $\ep$, 
then composing with the delta function and taking the generalized 
expectation give an asymptotic expansion of $p(\ep^2,x,x)$ in $\ep$. 

When proceeding in this approach, 
the remaining task to get over is 
to get the asymptotic expansion of $U^{\ep}_t$. 
On this point, Takanobu \cite{takanobu} gives a quite general result
taking advantage of 
the stochastic Taylor expansion.
The aim of the present paper is to show in the case of CR manifolds that
the process $U^{\ep}_t$ can be asymptotically expanded in a way that
is different from the stochastic Taylor expansion and
is more close to the geometric structure of CR manifolds. 
Our expansion takes advantage of the fact that,
in a certain local coordinate system introduced by Folland-Stein \cite{folland-stein},
behavior of CR manifolds can be seen as a perturbation of the 
Heisenberg group. 
To obtain the desired expansion, 
we will refine the result of \cite{folland-stein}, that is, 
we will determine all higher terms of the asymptotic behavior in this coordinate. 

The main result in this paper is that, 
in the asymptotic expansion of the heat kernel 
associated with $-\Delta_b/2$, 
the leading term depends only on the dimension of $M$, 
and coefficients of higher degree are represented in terms of 
the Takana-Webster connection on $M$. 
The result seems of interest in the point that
the coefficients of the asymptotic expansion can be
written in a concrete way with the expectation of 
Wiener functionals. 
This will be described through some examples. 

\bigbreak

In Section \ref{sec.diffusion}, 
we will give a review on CR manifolds, the Tanaka-Webster connection, and
the diffusion process on CR manifolds constructed in \cite{kondo-taniguchi}. 
Section \ref{sec.Folland-Stein} will be devoted to a study on the asymptotic behavior
of the local orthonormal frame written in the Folland-Stein local coordinate. 
We need information of higher terms of asymptotic expansion than shown in the 
original literature \cite{folland-stein}. 
Our main theorem which describes a relationship between the coefficient of the
asymptotic expansion of the heat kernel and the Tanaka-Webster connection
will be shown in Section \ref{sec:asymptotic expansion}. 
Finally in Section \ref{sec:ex} we will study examples of the asymptotic expansion, 
in the cases of the Heisenberg group and the CR sphere.

\section{Diffusion associated with the sub-Laplacian}\label{sec.diffusion}
We begin with recalling diffusion processes associated with sub-Laplacian
on CR manifolds constructed in \cite{kondo-taniguchi}. 

\subsection{CR manifolds}
We briefly review results on CR manifolds following \cite{dragomir-tomassini}. 
Let $M$ be an oriented strictly pseudoconvex CR manifold of 
dimension $2n+1$ with $n\in\{1,2,3,\ldots\}$, i.e.
$M$ is a real oriented manifold of dimension $2n+1$ together with
a complex subbundle $T_{1,0}$ of the complexified tangent bundle
$\mathbb{C}TM=TM\otimes_{\mathbb{R}}\mathbb{C}$ such that
$T_{1,0}\cap T_{0,1}=\{0\}$ and $[T_{1,0},T_{1,0}]\subset T_{1,0}$, 
where $T_{0,1}=\overline{T_{1,0}}$, 
and a real non-vanishing 1-form $\theta$ on $M$ which annihilates 
$H=\Re(T_{1,0}\oplus T_{0,1})$ such that the Levi form
\[
L_{\theta}(Z,W)=-\kyosu d\theta(Z,W), \qquad 
Z,W: \text{$C^{\infty}$ cross section of $T_{1,0}\oplus T_{0,1}$}
\]
is positive definite. 
Let $T$ be the characteristic direction, that is, 
the unique real vector field on $M$ such that
$T\rfloor d\theta=0$, $T\rfloor\theta=1$, where 
$T\rfloor \omega$ is the interior product. 

The real sub-Laplacian $\Delta_b$ on functions is defined by 
\[
\langle \Delta_b u,v\rangle_{\theta}=\langle d_bu,d_bv\rangle, \qquad
u,v\in C_0^{\infty}(M), 
\]
where $\langle,\rangle$ is the $L^2$-inner product on functions given by
\[
\langle u,v\rangle_{\theta}=\int_M u\overline{v}\psi, \qquad \psi=\theta\wedge(d\theta)^n, 
u,v\in C_0^{\infty}(M;\mathbb{C})\equiv\{f+\kyosu g\mid f,g\in C_0^{\infty}(M)\}, 
\]
and the section $d_bu$ for $u\in C^{\infty}(M)$ is defined by
the composite of $du$ with the natural restriction $r\colon T^*M\to H^*$.

\subsection{Tanaka-Webster connection}
We now review a connection on the $T_{1,0}$ due to 
Tanaka \cite{tanaka} and Webster \cite{webster}, following \cite{dragomir-tomassini}. 

Let $J\colon H\to H$ be the complex structure related to $(M,T_{1,0})$, 
which means that 
the $\mathbb{C}$-linear extension of $J$ is the 
multiplication by $\kyosu$ on $T_{1,0}$ and $-\kyosu$ on $T_{0,1}$
(note that $H\otimes_{\mathbb{R}}\mathbb{C}=T_{1,0}\oplus T_{0,1}$). 
We extend $J$ linearly to $TM$ by $J(T)=0$. 

Since $TM=H\oplus\mathbb{R}T=\{X+aT\mid X\in H, a\in\mathbb{R}\}$, there exists the unique
Riemannian metric $g_{\theta}$ on $M$ satisfying that
\[
g_{\theta}(X, Y)=d\theta(X,JY), \quad g_{\theta}(X,T)=0, \quad g_{\theta}(T,T)=1, 
\quad X, Y\in H.
\]
$g_{\theta}$ is called the Webster metric.
We extend $g_{\theta}$ to $\mathbb{C}TM$ $\mathbb{C}$-bilinearly. 

The Tanaka-Webster connection is the unique linear connection $\nabla$ on $M$
satisfying that 
\begin{align}
&\nabla_X Y\in\Gamma^{\infty}(H), \quad X\in\Gamma^{\infty}(TM), Y\in\Gamma^{\infty}(H), \nonumber\\
&\nabla J=0, \quad \nabla g_{\theta}=0, \nonumber\\
& T_{\nabla}(Z,W)=0, \quad Z,W\in \Gamma^{\infty}(T_{1,0}), \label{eqn.nabla.orth}\\
& T_{\nabla}(Z,W)=2\kyosu L_{\theta}(Z,W)T, \quad Z\in \Gamma^{\infty}(T_{1,0}), W\in \Gamma^{\infty}(T_{0,1}),\label{eqn.nabla.norm}\\
& T_{\nabla}(T,J(X))+J(T_{\nabla}(T,X))=0, \quad X\in \Gamma^{\infty}(TM),\nonumber
\end{align}
where $\nabla_X$ is the covariant derivative in the direction of $X$ and 
$T_{\nabla}$ is the torsion tensor field of $\nabla$: 
$T_{\nabla}(Z,W)=\nabla_{Z}W-\nabla_{W}Z-[Z,W]$. 

Let $\{Z_{\alpha}\}_{\alpha=1,\ldots,n}$ be a local
orthonormal frame for $T_{1,0}$ on an open set $U$, that is, 
$Z_{\alpha}$ is a $T_{1,0}$-valued section defined on $U$ and 
$g_{\theta}(Z_{\alpha},Z_{\overline{\beta}})=\delta_{\alpha\beta}$, where 
$Z_{\overline{\beta}}=\overline{Z_{\beta}}$.  
If we set $Z_0=T$, 
$\{Z_A\}_{A=0,1,\ldots,n,\overline{1},\ldots,\overline{n}}$ is a local frame for $\mathbb{C}TM$. 
We define Christoffel symbols 
$\Gamma_{AB}^C$ for 
$A, B, C=0,1,\ldots,n,\overline{1},\ldots,\overline{n}$ by
\[
\nabla_{Z_A}Z_B=\sum_{C\in\{0,1,\ldots,n,\overline{1},\ldots,\overline{n}\}}\Gamma_{AB}^C Z_C.
\]
Christoffel symbols satisfy the following: 
\begin{align}
&\Gamma_{AB}^C=0, \quad \text{unless 
$(B,C)\in\{(\beta,\gamma),(\overline{\beta},\overline{\gamma});\beta,\gamma=1,\ldots,n\}$},\\
&\Gamma_{A\beta}^{\gamma}+\Gamma_{A\overline{\gamma}}^{\overline{\beta}}=0,\quad
\beta,\gamma=1,\ldots,n, \quad A=0,1,\ldots,n,\overline{1},\ldots,\overline{n}. \label{eqn.nabla.antisym}
\end{align}

\subsection{Diffusion process}
A diffusion process $\mathbb{X}=\{(\{X(t)\}_{t\ge0},P_x);x\in M\}$ generated by $-\Delta_b/2$
is constructed in \cite{kondo-taniguchi} via the Eells-Elworthy-Malliavin method. 
To be more precise, 
let $U(T_{1,0})$ be the complex unitary bundle over $M$ given by
\[
U(T_{1,0})_x=\{r\colon\mathbb{C}^n\to(T_{1,0})_x;\text{$r$ is an isometry}\}, 
\]
and $\pi\colon U(T_{1,0})\to M$ be the natural projection. 
Here $\mathbb{C}^n$ is considered to be equipped by the standard metric and 
$T_{1,0}$ by the Webster metric. 

The canonical vector fields $\{L_{\alpha}\}_{\alpha=1,\ldots,n}$ on $U(T_{1,0})$
associated with the Tanaka-Webster connection can be defined via 
horizontal lift (see \cite{kondo-taniguchi} for detail). 

Let $\{B^1_t,\ldots,B^{2n}_t\}_{t\ge0}$ be a $2n$-dimensional Brownian motion and 
let 
\[
B^{\alpha}_t=\frac{1}{\sqrt{2}}(B^{\alpha}_t+\kyosu B^{n+\alpha}_t), \qquad
B^{\overline{\alpha}}_t=\frac{1}{\sqrt{2}}(B^{\alpha}_t-\kyosu B^{n+\alpha}_t),\quad
\alpha=1,\ldots,n.
\]
The SDE on $U(T_{1,0})$ 
\[
dr_t=\sum_{\alpha=1}^n L_{\alpha}(r_t)\circ dB^{\alpha}_t
+\sum_{\alpha=1}^n\overline{L_{\alpha}}(r_t)\circ dB^{\overline{\alpha}}_t, \qquad
r_0=r\in U(T_{1,0}), 
\]
or equivalently, 
\[
dr_t=
\sum_{\alpha=1}^n \sqrt{2}\Re L_{\alpha}(r_t)\circ dB^{\alpha}_t
+\sum_{\alpha=1}^n \sqrt{2}\Im L_{\alpha}(r_t)\circ dB^{n+\alpha}_t, \qquad
r_0=r\in U(T_{1,0}), 
\]
has a unique strong solution $\{r_t=r(t,r,B)\}_{t\ge0}$. 
Here the symbol $\circ$ stands for the Stratonovich integral.

Projecting the process $r_t$ on $M$ yields the 
desired diffusion process $\mathbb{X}$, because
for any $x\in M$, 
the induced probability measures on the space of continuous functions 
$[0,\infty)\to M$ coincide for all $r\in\pi^{-1}(x)$. 

We note here a local representation of $-\Delta_b/2$. 
Let $\{Z_{\alpha}\}_{\alpha=1,\ldots,n}$ be a local orthonormal frame for $T_{1,0}$ on
an open set $U$ of $M$, and
$\Gamma_{AB}^C$ the associated Christoffel symbols as in the previous subsection. 
Then a straightforward calculation shows that 
$-\Delta_b/2$ is represented locally as follows:
\[
-\frac{1}{2}\Delta_b
=\frac{1}{2}\sum_{\alpha=1}^n(Z_{\alpha}Z_{\overline{\alpha}}+Z_{\overline{\alpha}}Z_{\alpha})
-\frac{1}{2}\sum_{\alpha,\beta=1}^n
(\Gamma_{\overline{\beta}\beta}^{\alpha}Z_{\alpha}
+\Gamma_{\beta\overline{\beta}}^{\overline{\alpha}}Z_{\overline{\alpha}}). 
\]

\section{Folland-Stein normal coordinate}\label{sec.Folland-Stein}
Let $\{Z_{\alpha}\}_{\alpha=1,\ldots,n}$ be a local orthonormal frame around 
a fixed point $x\in M$. 
We set
\[
\XH_{\alpha}
=\sqrt{2}\Re Z_{\alpha}=\frac{1}{\sqrt{2}}(Z_{\alpha}+Z_{\overline{\alpha}}),
\qquad
\XH_{n+\alpha}
=\sqrt{2}\Im Z_{\alpha}=-\frac{\kyosu}{\sqrt{2}}(Z_{\alpha}-Z_{\overline{\alpha}})
\]
for $\alpha=1,\ldots,n$ and $\XH_{2n+1}=T$. 
Given $u=(u^1,\ldots,u^{2n+1})\in\mathbb{R}^{2n+1}$, 
let $C_u\colon[0,1]\to M$ be the integral curve of the tangent vector field
\[
\XH_u=\sum_{j=1}^{2n+1}u^j\XH_j
\]
starting at $x$, i.e. $C_u$ is defined by
\[
\frac{dC_u}{dt}(t)=\XH_u(C_u(t)), \qquad C_u(0)=x. 
\]
We set $E_x(u)=C_u(1)$. 
By a standard theory of ordinary differential equations, 
$E_x$ defines a diffeomorphism of a neighborhood $U_x$ of $0\in\mathbb{R}^{2n+1}$ onto
a neighborhood $V_x$ of $x\in M$. 
The inverse map $E_x^{-1}\colon V_x\to U_x$ defines a local chart of $M$ and
the resulting local coordinates $u^j\colon V_x\to\mathbb{R}$, $j=1,\ldots,2n+1$, 
are referred to as the Folland-Stein normal coordinates. 
We hereinafter regard a function on $V_x$ as a function of $u\in U_x$ via $E_x$. 

To describe an asymptotic behavior of vector fields with respect to these coordinates, 
the following asymptotic notation is introduced: 

\begin{definition}
For $a\ge1$, a function $f$ on $V_x$ is said to be $O^a$ if
$f(u)=O\bigl(\bigl(\sum_{j=1}^{2n}\lvert u^j\rvert+\lvert u^{2n+1}\rvert^2\bigr)^a\bigr)$ as 
$u=(u^1,\ldots,u^{2n+1})\to 0$. We may simply write $f=O^a$. 
\end{definition}

Define the structure functions $C^i_{jk}\colon U_x\to\mathbb{R}$, $i,j,k=1,\ldots,2n+1$ by
\[
[\XH_j,\XH_k]=\sum_{i=1}^{2n+1}C^i_{jk}\XH_i. 
\]

\begin{theorem}\label{thm:F-S}
$\XH_j$, $j=1,\ldots,2n$ can be written with respect to the Folland-Stein
normal coordinates $(u^j)_{j=1,\ldots,2n+1}$ as follows: 
\begin{align*}
\XH_{\alpha}
&=\frac{\pa}{\pa u^{\alpha}}-u^{n+\alpha}\frac{\pa}{\pa u^{2n+1}}\\
&\qquad
+\sum_{i=1}^{2n}
\Bigl(\sum_{1\le\lVert J\rVert\le a}c_{\alpha,i,J}u^J+O^{a+1}\Bigr)\frac{\pa}{\pa u^i}
+\Bigl(\sum_{2\le\lVert J\rVert\le a}c_{\alpha,2n+1,J}u^J+O^{a+1}\Bigr)\frac{\pa}{\pa u^{2n+1}},\\
\XH_{n+\alpha}
&=\frac{\pa}{\pa u^{n+\alpha}}+u^{\alpha}\frac{\pa}{\pa u^{2n+1}}\\
&\qquad
+\sum_{i=1}^{2n}
\Bigl(\sum_{1\le\lVert J\rVert\le a}c_{n+\alpha,i,J}u^J+O^{a+1}\Bigr)\frac{\pa}{\pa u^i}
+\Bigl(\sum_{2\le\lVert J\rVert\le a}c_{n+\alpha,2n+1,J}u^J+O^{a+1}\Bigr)\frac{\pa}{\pa u^{2n+1}}
\end{align*}
for $\alpha=1,\ldots,n$. 
Moreover, each $c_{j,i,J}$, $i=1,\ldots,2n+1$, $j=1,\ldots,2n$, $\lVert J\rVert=a$, 
is an at most $a$-th degree polynomial of 
$\frac{d^{b}}{ds^{b}}C^{\alpha}_{\beta\gamma}(su)|_{s=0}$ for $b\le a$.
\end{theorem}

To prove the theorem, write
\[
\XH_k=\sum_{j=1}^{2n+1}G^j_k\frac{\partial}{\partial u^j},\quad k=1,\ldots,2n
\]
using the Folland-Stein normal coordinates, 
where $G^j_k$ is a smooth function on $U_x$. 
Note that $G^j_k(0)=\delta_{jk}$ by definition of $E_x$. 
Since $G(u)=(G^j_k(u))_{j,k=1,\ldots,2n+1}$ is an invertible matrix 
for every $u\in U_x$, we can define 
$F(u)=(F^j_k(u))_{j,k=1,\ldots,2n+1}$, $F^j_k\colon U_x\to\mathbb{R}$ by
$F=G^{-1}$. 

Let us consider the matrices 
$\mathcal{A}=(\mathcal{A}^j_k)_{j,k=1,\ldots,2n+1}$, 
$\Xi=(\Xi^j_k)_{j,k=1,\ldots,2n+1}$, 
$\mathcal{A}^j_k, \Xi^j_k\colon [-1,1]\times U_x\to\mathbb{R}$ defined by
\[
\mathcal{A}^j_k(s,u)=sF^j_k(su),\qquad
\Xi^j_k(s,u)=\sum_{l=1}^{2n+1}C^j_{lk}(su)u^l. 
\]
We use the following lemma \cite[Lemma 3.3]{dragomir-tomassini}:

\begin{lemma}\label{lem:DT}
We have
\[
\frac{\pa}{\pa s}\mathcal{A}=I-\Xi\mathcal{A},
\]
where $I$ is the identity matrix of size $2n+1$. 
\end{lemma}

By Taylor's theorem, we can write for $A=0,1,\ldots,$ 
\[
G^j_k(su)=\sum_{a=0}^A\frac{s^a}{a!}
\sum_{\lvert J\rvert=a}
\frac{\pa^a G^j_k}{\pa u^J}(0)u^J
+\frac{s^{a+1}}{(a+1)!}
\sum_{\lvert J\rvert=a+1}
\frac{\pa^{a+1} G^j_k}{\pa u^J}(csu)u^J
\]
for some $c\in(0,1)$. 
Here we use the notation $\lvert J\rvert=a$ and $u^J=u^{j_1}\cdots u^{j_a}$ for
a multiple index $J=(j_1,\ldots,j_a)$, $j_1,\ldots,j_a\in\{1,\ldots,2n+1\}$. 
Since 
\[
\frac{\pa^{a+1} G^j_k}{\pa u^J}(csu)u^J=O^{a+1}
\]
for each $J$ with $\lvert J\rvert=a+1$, 
we have
\[
G^j_k(u)=\sum_{a=0}^A G^{(a)}(u)^j_k+O^{a+1}, 
\]
where
\[
G^{(a)}(u)^j_k
=\frac{1}{a!}
\sum_{\lvert J\rvert=a}
\frac{\pa^a G^j_k}{\pa u^J}(0)u^J. 
\]
To compute $G^{(a)}(u)=(G^{(a)}(u)^j_k)_{j,k=1,\ldots,2n+1}$, 
we note that the Taylor expansion yields
\[
G(su)=\sum_{a=0}^{\infty}s^aG^{(a)}(u), \qquad
F(su)=\sum_{a=0}^{\infty}s^aF^{(a)}(u)
\]
with $G^{(0)}(u)=F^{(0)}(u)=I$ and 
\begin{equation}\label{eq:GF=I}
\sum_{b=0}^a G^{(b)}(u)F^{(a-b)}(u)=0,  \quad a\ge1, 
\end{equation}
because $GF=I$. 
We have by definition of $\mathcal{A}$ an expansion
\[
\mathcal{A}(s,u)=\sum_{a=0}^{\infty}s^{a+1}F^{(a)}(u). 
\]
By Lemma \ref{lem:DT}, we have
\begin{equation}\label{eq:F_Gamma}
(a+1)F^{(a)}(u)=-\sum_{b=0}^{a-1}\Xi^{(b)}(u)F^{(a-b-1)}(u), \quad
a\ge1.
\end{equation}
We have from \eqref{eq:GF=I} and \eqref{eq:F_Gamma} that 
\[
G^{(a)}(u)=
\sum_{b_1+\cdots+b_i=a}c_{b_1,\ldots,b_i}\Xi^{(b_1-1)}(u)\cdots\Xi^{(b_i-1)}(u)
\]
for some constants $c_{b_1,\ldots,b_i}$. 
Therefore we have the following expression:

\begin{proposition}\label{prop:G_asymp}
Each component $G^{(a)}(u)^j_k$ can be written 
as a linear combination of terms
\[
\prod_{q=1}^r\frac{d^{b_q}}{ds^{b_q}}C^{\alpha_q}_{\beta_q\gamma_q}(su)\Bigm|_{s=0}
u^{m_1}\cdots u^{m_i}
\]
with $b_1+\cdots+b_r+m_1+\cdots+m_i=a$. 
For small $a$ we have
\begin{align*}
G^{(1)}(u)^j_k&=\frac{1}{2}\sum_{l=1}^{2n+1}C^j_{lk}(0)u^l, \\
G^{(2)}(u)^j_k&=\frac{1}{12}\sum_{l,m,p=1}^{2n+1}C^j_{lp}(0)C^p_{mk}(0)u^lu^m
+\frac{1}{3}\sum_{l=1}^{2n+1}\frac{d}{ds}C^j_{lk}(su)\Bigm|_{s=0}u^l, \\
G^{(3)}(u)^j_k&=
\frac{1}{8}\sum_{l,m,p=1}^{2n+1}C^j_{lp}(0)\frac{d}{ds}C^p_{mk}(su)\Bigm|_{s=0}u^lu^m
+\frac{1}{8}\sum_{l=1}^{2n+1}\frac{d^2}{ds^2}C^j_{lk}(su)\Bigm|_{s=0}u^l. 
\end{align*}
\end{proposition}

To establish Theorem \ref{thm:F-S}, 
it remains to show that 
\[
G^{(1)}(u)^{\alpha}_{2n+1}=-u^{n+\alpha}+O^2,\qquad
G^{(1)}(u)^{n+\alpha}_{2n+1}=u^{\alpha}+O^2
\]
for $\alpha=1,\ldots,n$. To this end, it is enough to show the following:

\begin{proposition}[\cite{dragomir-tomassini}]\label{prop:str_const}
For $\alpha,\beta=1,\ldots,n$ we have 
\[
C^{2n+1}_{\alpha\beta}=C^{2n+1}_{n+\alpha,n+\beta}=0,\qquad
C^{2n+1}_{\alpha,n+\beta}=-C^{2n+1}_{n+\alpha,\beta}=2\delta_{\alpha\beta}
\]
on $U_x$. 
\end{proposition}
\begin{proof}
By \eqref{eqn.nabla.orth}, we have
\[
0=T_{\nabla}(Z_{\alpha},Z_{\beta})
=\nabla_{Z_{\alpha}}Z_{\beta}-\nabla_{Z_{\beta}}Z_{\alpha}-[Z_{\alpha},Z_{\beta}]. 
\]
This is rewritten using Christoffel symbols as
\[
[Z_{\alpha},Z_{\beta}]
=\sum_{\gamma=1}^{n}(\Gamma^{\gamma}_{\alpha\beta}-\Gamma^{\gamma}_{\beta\alpha})Z_{\gamma}. 
\]
Taking conjugate of this equality yields
\[
[Z_{\overline{\alpha}},Z_{\overline{\beta}}]
=\sum_{\gamma=1}^{n}(\Gamma^{\overline{\gamma}}_{\overline{\alpha}\overline{\beta}}
-\Gamma^{\overline{\gamma}}_{\overline{\beta}\overline{\alpha}})Z_{\overline{\gamma}}. 
\]
Next we have from \eqref{eqn.nabla.norm} that 
\[
2\kyosu L_{\theta}(Z_{\alpha},Z_{\overline{\beta}})T
=T_{\nabla}(Z_{\alpha},Z_{\overline{\beta}})
=\nabla_{Z_{\alpha}}Z_{\overline{\beta}}
-\nabla_{Z_{\overline{\beta}}}Z_{\alpha}-[Z_{\alpha},Z_{\overline{\beta}}]. 
\]
Together with $L_{\theta}(Z_{\alpha},Z_{\overline{\beta}})=\delta_{\alpha\beta}$, we have
\begin{align}
[Z_{\alpha},Z_{\overline{\beta}}]
&=
-\sum_{\gamma=1}^n \Gamma^{\gamma}_{\overline{\beta}\alpha}Z_{\gamma}
+\sum_{\gamma=1}^n \Gamma^{\overline{\gamma}}_{\alpha\overline{\beta}}Z_{\overline{\gamma}}
-2\sqrt{-1}\delta_{\alpha\beta}T, \label{eq:Z_bracket}\\
[Z_{\overline{\alpha}},Z_{\beta}]
&=
\sum_{\gamma=1}^n \Gamma^{\gamma}_{\overline{\alpha}\beta}Z_{\gamma}
-\sum_{\gamma=1}^n \Gamma^{\overline{\gamma}}_{\beta\overline{\alpha}}Z_{\overline{\gamma}}
+2\sqrt{-1}\delta_{\alpha\beta}T.\nonumber
\end{align}
The proposition follows by rewriting these equalities using definition of $\XH_j$. 
\end{proof}

Theorem \ref{thm:F-S} has been proved 
from Propositions \ref{prop:G_asymp} and \ref{prop:str_const}. 

\begin{remark}
We note that an idea using exponential maps associated with vector fields is 
also seen in Ben Arous \cite{benarous}. 

One of difference between \cite{benarous} and the present paper is 
the way in which exponential maps are used. 
In \cite{benarous}, exponential maps are used to transform the solution of an SDE 
associated with vector fields. 
On the other hand, we have represented the vector fields in the 
Folland-Stein coordinate which is obtained by exponential maps, 
then in the next section we will consider an SDE in this coordinate
and proceed in a perturbation argument taking advantage of 
Theorem \ref{thm:F-S}. 
\end{remark}

\section{Asymptotic expansion of the heat kernel}\label{sec:asymptotic expansion}
In this section, we consider the heat equation
\[
\frac{\pa}{\pa t}v=-\frac{1}{2}\Delta_b u, \quad
v(0,x)=f(x), \quad f\in C_b^{\infty}(M). 
\]
We assume that $M$ is compact. 
Then we have by \cite{kondo-taniguchi} that 
there is a function $p\in C^{\infty}((0,\infty)\times M\times M)$ such that 
\[
P_x(X(t)\in dy)=p(t,x,y)\psi(dy),\quad x,y\in M. 
\]
$p(t,x,y)$ is the heat kernel associated to $-\frac{1}{2}\Delta_b$. 

Now we state the main theorem of the paper: 

\begin{theorem}\label{thm:main}
For each $x\in M$, there exists constants $c^M_a(x)$ ($a=0,1,\ldots$) such that
\[
p(t,x,x)\sim t^{-n-1}\sum_{a=0}^{\infty}c^M_a(x)t^a\quad\text{as $t\downarrow 0$}. 
\]
Furthermore, 
\begin{enumerate}
\item $c^M_0(x)=c_0=
\frac{1}{(2\pi)^{n+1}}\int_{\mathbb{R}}(\frac{2\tau}{\sinh 2\tau})^n d\tau$, 
which depends only on $n$. 
\item For $a\ge1$, $c^M_a(x)$ can be written as an at most $2a$-th degree polynomial of 
at most $2a+2$-th derivatives of Christoffel symbols of $M$. 
\end{enumerate}
\end{theorem}

To prove the theorem, we first note that it is sufficient to 
consider the problem locally. 
To be precise, let $\{Z_{\alpha}\}_{\alpha=1,\ldots,n}$ be a local orthonormal frame
for $T_{1,0}$ on a relatively compact open neighborhood $V$ of $x$. 
Let $\XH_j$, $j=1,\ldots,2n$ be as in Section \ref{sec.Folland-Stein} and 
\[
\XH_0
=-\sum_{\alpha,\beta=1}^n
(\Gamma^{\alpha}_{\overline{\beta}\beta}Z_{\alpha}
+\Gamma^{\overline{\alpha}}_{\beta\overline{\beta}}Z_{\overline{\alpha}})
=-\sqrt{2}\sum_{\alpha,\beta=1}^n
(\Re\Gamma^{\alpha}_{\overline{\beta}\beta}\XH_{\alpha}
-\Im\Gamma^{\alpha}_{\overline{\beta}\beta}\XH_{n+\alpha}). 
\]

Let $\{((U_t)_{t\ge0},\widetilde{P}_y);y\in V\}$ be the minimal diffusion process on $V$ 
generated by 
\[
\mathcal{L}=\frac{1}{2}\sum_{j=1}^{2n}\XH_j^2+\XH_0
\]
and  $\widetilde{p}$ be the density function of $\widetilde{P}_y$ with respect to $\psi$. 
Repetition of the argument in \cite[Section V.10]{i-w} yields the following:

\begin{lemma}\label{lem:localization}
There exist constants $c_1,c_2>0$ such that 
$\lvert p(t,x,x)-\widetilde{p}(t,x,x)\rvert\le c_1\exp(-\frac{c_2}{t})$, $t>0$. 
\end{lemma}

By Lemma \ref{lem:localization}, 
for a proof of Theorem \ref{thm:main} it suffices to show the theorem replacing $p$ by
the heat kernel associated with $\mathcal{L}$.
Moreover, by Theorem \ref{thm:F-S} we can assume that $\XH_j$'s are 
vector fields on $\mathbb{R}^{2n+1}$ which satisfy the condition of Theorem \ref{thm:F-S}
and $x=0\in\mathbb{R}^{2n+1}$. 
For simplicity we write $p$ by the heat kernel associated with $\mathcal{L}$ in the sequel. 

We write 
\[
\XH_j=X_j+A_j,\quad j=1,\ldots,2n,
\]
where
\[
X_{\alpha}=\frac{\pa}{\pa u^{\alpha}}-u^{n+\alpha}\frac{\pa}{\pa u^{2n+1}}, \qquad
X_{n+\alpha}=\frac{\pa}{\pa u^{n+\alpha}}+u^{\alpha}\frac{\pa}{\pa u^{2n+1}}, \quad
\alpha=1,\ldots,n
\]
and
\[
A_j=\sum_{i=1}^{2n+1}a^i_j\frac{\pa}{\pa u^i}, \quad j=1,\ldots,2n. 
\]
We also write
\[
\XH_0=\sum_{i=1}^{2n+1}b^i\frac{\pa}{\pa u^i}. 
\]
For later use, for $\ep\in(0,1]$ we define 
\begin{equation}\label{eq:XHepj}
\XH^{\ep}_j=X_j+A^{\ep}_j,\qquad 
A^{\ep}_j=\sum_{i=1}^{2n+1}a^{i,\ep}_j\frac{\pa}{\pa u^i}, \quad i=1,\ldots,2n,
\end{equation}
\begin{equation}\label{eq:XHep0}
\XH^{\ep}_0=\sum_{i=1}^{2n+1}b^{i,\ep}\frac{\pa}{\pa u^i}
\end{equation}
with
\[
a^{i,\ep}_j(u)=a^i_j(\lambda_{\ep}(u)), \qquad
a^{2n+1,\ep}_j(u)=\ep^{-1}a^{2n+1}_j(\lambda_{\ep}(u)), \quad
i=1,\ldots,2n,\quad j=1,\ldots,2n,
\]
\[
b^{i,\ep}(u)=\ep b^i(\lambda_{\ep}(u)), \qquad
b^{2n+1,\ep}(u)=b^{2n+1}(\lambda_{\ep}(u)), \quad i=1,\ldots,2n,
\]
where $\lambda_{\ep}(u)=(\ep u^1,\ldots,\ep u^{2n},\ep^2u^{2n+1})$ for
$u=(u^1,\ldots,u^{2n+1})$. 
Note that 
\[
\XH^{\ep}_j=\XH_j, \quad j=1,\ldots,2n+1 \quad \text{if $\ep=1$}. 
\]
Then we have from Theorem \ref{thm:F-S} that
\begin{equation}
a^{i,\ep}_j(0)=0, \quad i=1,\ldots,2n+1, \quad j=1,\ldots,2n, \quad \ep\in(0,1],
\label{eq:aij}
\end{equation}
\begin{equation}
\frac{\pa a^{2n+1,\ep}_j}{\pa u^k}(0)=0, \quad j=1,\ldots,2n, 
\quad k=1,\ldots,2n+1, \quad \ep\in(0,1]. \label{eq:aij'}
\end{equation}

To state the next lemma, 
let $S^{2n}=\{\xi\in\mathbb{R}^{2n+1};\lvert\xi\rvert=1\}$ be the $2n$-dimensional sphere and 
denote the standard inner product in $\mathbb{R}^{2n+1}$ by 
$\langle ,\rangle_{\mathbb{R}^{2n+1}}$.
We identify $T_0\mathbb{R}^{2n+1}$ and $\mathbb{R}^{2n+1}$ by
$\sum_{j=1}^{2n+1}x^j(\frac{\pa}{\pa u^j})_0\mapsto (x^j)_{j=1,\ldots,2n+1}$.

\begin{lemma}\label{lem:Hormander}
For $\alpha=1,\ldots,n$, 
we have 
\[
\inf_{\ep\in(0,1]}\inf_{\xi\in S^{2n}}
\biggl(
\sum_{j=1}^{2n}\langle(\XH^{\ep}_j)_0,\xi\rangle_{\mathbb{R}^{2n+1}}^2
+\langle[\XH^{\ep}_{\alpha},\XH^{\ep}_{n+\alpha}]_0,\xi\rangle_{\mathbb{R}^{2n+1}}^2
\biggr)>0,
\]
where
the tangent vector at $0$ of a tangent vector field $V$ is denoted by $V_0$. 
In particular, 
the set $\{(\XH_j)_0;i=1,\ldots,2n\}\cup\{[\XH_{\alpha},\XH_{n+\alpha}]_0\}$ spans
$T_0\mathbb{R}^{2n+1}$ over $\mathbb{R}$, 
\end{lemma}
\begin{proof}
We have $(\XH_j)_0=(\frac{\pa}{\pa u^j})_0$ for $i=1,\ldots,2n$ by \eqref{eq:aij}. 
if we write
\[
\eta^i_{\alpha}=
\frac{\pa a^i_{n+\alpha}}{\pa u^{\alpha}}(0)
-\frac{\pa a^i_{\alpha}}{\pa u^{n+\alpha}}(0), \quad i=1,\ldots,2n,
\]
we have
\begin{align*}
[\XH_{\alpha},\XH_{n+\alpha}]_0
&=2\biggl(\frac{\pa}{\pa u^{2n+1}}\biggr)_0
+\ep\sum_{i=1}^{2n}\eta^i_{\alpha}\biggl(\frac{\pa}{\pa u^i}\biggr)_0\\
&\equiv 2\biggl(\frac{\pa}{\pa u^{2n+1}}\biggr)_0 \quad 
\bmod \mathrm{span}_{\mathbb{R}}\{(\XH_j)_0;j=1,\ldots,2n\}
\end{align*}
by \eqref{eq:aij'}. Then 
\[
\sum_{j=1}^{2n}\langle(\XH^{\ep}_j)_0,\xi\rangle_{\mathbb{R}^{2n+1}}^2
+\langle[\XH^{\ep}_{\alpha},\XH^{\ep}_{n+\alpha}]_0,\xi\rangle_{\mathbb{R}^{2n+1}}^2
=\sum_{j=1}^{2n}(\xi^j)^2+
\biggl(
2\xi^{2n+1}+\ep\sum_{j=1}^{2n}\eta^j_{\alpha}\xi^j
\biggr)^2
\]
can be seen as a continuous function on the compact set
$S^{2n}\times[0,1]\ni(\xi,\ep)$ which never attains zero. 
The lemma follows. 
\end{proof}

For $\ep>0$, let $\{U^{\ep}_t\}_{t\ge0}$ be the solution of SDE
\[
dU^{\ep}_t
=\ep\sum_{j=1}^{2n}\XH_j(U^{\ep}_t)\circ dB^j_t+\ep^2\XH_0(U^{\ep}_t)dt, 
\quad U^{\ep}_0=0, 
\]
where $(B^1_t,\ldots,B^{2n}_t)_{t\ge0}$ is a Brownian motion of $2n$ dimension.
Then $\{U^1_t\}_{t\ge0}$ is a diffusion process with generator $\mathcal{L}$. 

Since $\XH_0,\XH_1,\ldots,\XH_{2n}$ satisfies the H\"{o}rmander condition
by Lemma \ref{lem:Hormander}, 
we have $U^{\ep}_t\in\mathbb{D}^{\infty}(\mathbb{R}^{2n+1})$ , i.e.
$U^{\ep}_t$ is nondegenerate in the sense of the Malliavin calculus. 
Moreover, $U^{\ep}_t$ and $U^{1}_{\ep^2 t}$ have the same law. 
Then we can define $\delta_0(U^{\ep}_t)$ as a generalized Wiener functionals and 
we have
\[
p(\ep^2,0,0)=\mathbb{E}[\delta_0(U^{\ep}_t)]=\mathbb{E}[\delta_0(U^{1}_{\ep^2 t})], 
\]
where the expectation is denoted by $\mathbb{E}$ and $\delta_0$ is the 
Dirac delta function at $0\in\mathbb{R}^{2n+1}$. 

To investigate the asymptotic behavior of $\mathbb{E}[\delta_0(U^{1}_{\ep^2 t})]$, 
we first introduce the multiple Wiener integral: 

\begin{definition}
For a multiple index $J=(j_1,\ldots,j_b)$, $j_1,\ldots,j_b\in\{0,1,\ldots,2n\}$, 
define $B^J_t$ by
\[
B^{(j_1)}_t=B^{j_1}_t, \qquad
B^{(j_1,\ldots,j_b)}_t=\int_0^t B^{(j_1,\ldots,j_{b-1})}_s\circ dB^{j_b}_s\quad
\text{for $b\ge2$}, 
\]
where we set $B^0_t=t$. 
\end{definition}

By applying the It\^{o} formula repetitively, we have the following expansion 
formula (cf. \cite{i-w}): 

\begin{proposition}
For $i=1,\ldots,2n+1$, $a=1,2,\ldots$ and $\ep>0$, we have
\[
U^{\ep,i}_t
=\sum_{a=1}^A \ep^a\sum_{\lVert J\rVert=a}(\XH_Ju^i)(0)B^J_t
+\ep^{A+1}\Phi^{A,\ep}_t, 
\]
where we write $U^{\ep}_t=(U^{\ep,1}_t,\ldots,U^{\ep,2n+1}_t)$ and
\[
\lVert J\rVert=b+\#\{\nu;j_{\nu}=0\}, \qquad 
\XH_J=\XH_{j_1}\cdots\XH_{j_b}
\]
for $J=(j_1,\ldots,j_b)$, $j_1,\ldots,j_b\in\{0,1,\ldots,2n+1\}$. 

Moreover, we have
$\Phi^{A,\ep}_t\in\mathbb{D}^{\infty}(\mathbb{R}^{2n+1})$ and 
\[
\sup_{\ep\in(0,1)}\mathbb{E}\biggl[\sup_{t\in[0,1]}\lvert \Phi^{A,\ep}_t\rvert^p\biggr]<\infty,
\quad p\in(1,\infty). 
\]
\end{proposition}

Let us write $\Psi^{A,\ep}_t=(\Psi^{A,\ep,1}_t,\ldots,\Psi^{A,\ep,2n+1}_t)$ with
\[
\Psi^{A,\ep,i}_t=
\begin{cases}
\Phi^{A,\ep,i}_t, & i=1,\ldots,2n+1,\\
\Phi^{A+1,\ep,2n+1}_t, & i=2n+1. 
\end{cases}
\]
Using Theorem \ref{thm:F-S}, we have
\[
\sum_{\lVert J\rVert=1}(\XH_Ju^i)(0)B^J_t
=\sum_{j=1}^{2n}(\XH_ju^i)(0)B^j_t
=\begin{cases}
B^i_t, & i=1,\ldots,2n, \\
0, & i=2n+1
\end{cases}
\]
and 
\[
\sum_{\lVert J\rVert=2}(\XH_Ju^{2n+1})(0)B^J_t
=\sum_{j,k=1}^{2n}(\XH_j\XH_ku^{2n+1})(0)B^{(j,k)}_t
+(\XH_0u^{2n+1})(0)t
=\sum_{\alpha=1}^nS^{\alpha}_t, 
\]
where
\[
S^{\alpha}_t=B^{(\alpha,n+\alpha)}_t-B^{(n+\alpha,\alpha)}_t
=\int_0^t(B^{\alpha}_s\circ dB^{n+\alpha}_s-B^{n+\alpha}_s\circ dB^{\alpha}_s)
\]
is Levi's stochastic area. 

From this we can write 
\begin{align*}
U^{\ep,i}_t 
&=\ep B^i_t+\sum_{a=2}^A\ep^a\varphi^{a,i}_t+\ep^{A+1}\Psi^{A,\ep,i}_t,\quad i=1,\ldots,2n,\\
U^{\ep,2n+1}_t 
&=\ep^2 \sum_{\alpha=1}^n S^{\alpha}_t
+\sum_{a=2}^A\ep^{a+1}\varphi^{a,2n+1}_t+\ep^{A+2}\Psi^{A,\ep,i}_t,
\end{align*}
where
\begin{equation}
\varphi^{a,i}_t
=\begin{cases}
\sum_{\lVert J\rVert=a}(\XH_Ju^i)(0)B^J_t, & i=1,\ldots,2n,\\
\sum_{\lVert J\rVert=a}(\XH_Ju^{2n+1})(0)B^J_t, & i=2n+1.
\end{cases}\label{eq:varphi}
\end{equation}
Let us write $\mathbb{X}_t=(B^1_t,\ldots,B^{2n}_t,\sum_{\alpha=1}^n S^{\alpha}_t)$. 
By the change of variables
\[
(\ep u^2,\ldots,\ep u^{2n},\ep^2u^{2n+1})\mapsto(v^1,\ldots,v^{2n},v^{2n+1}), 
\]
we have
\[
\mathbb{E}[\delta_0(U^{\ep}_1)]
=\ep^{-2n-2}\mathbb{E}[\delta_0(\mathbb{X}_1+\ep\Psi^{1,\ep}_1)]. 
\]

Let $F^{\ep}_t=\mathbb{X}_t+\ep\Psi^{1,\ep}_t$ and 
let $\sigma^{\ep}$ be the Malliavin covariance of $F^{\ep}_1$. 

\begin{lemma}\label{lem:nondegenerate}
For any $p\in[1,\infty)$,
we have $\sup_{\ep\in(0,1]}\mathbb{E}[(\det\sigma^{\ep})^{-p}]<\infty$, 
i.e. $\{F^{\ep}_1\}_{\ep\in(0,1]}$ is uniformly nondegenerate. 
\end{lemma}
\begin{proof}
We write $F^{\ep}_t=(F^{\ep,i}_t)_{i=1,\ldots,2n+1}$. Since
\[
F^{\ep,i}_t=\ep^{-1}U^{\ep,i}_t, \qquad F^{\ep,2n+1}_t=\ep^{-2}U^{\ep,2n+1}_t,
\quad i=1,\ldots,2n,
\]
we can see that $\{F^{\ep}_t\}_{t\ge0}$ obeys the SDE
\[
dF^{\ep}_t=\sum_{j=1}^{2n}\XH^{\ep}_j(F^{\ep}_t)\circ dB^j_t
+\XH^{\ep}_0(F^{\ep}_t)dt, 
\]
where $\XH^{\ep}_j$ is given in \eqref{eq:XHepj}, \eqref{eq:XHep0}. 

Note that for any $a\ge1$, $j_1,\ldots,j_a\in\{0,\ldots,2n\}$, 
the vector field
\[
[\XH^{\ep}_{j_a},[\ldots,[\XH^{\ep}_{j_2},\XH^{\ep}_{j_1}]\ldots]]
\]
has bounded coefficients uniformly in $\ep\in(0,1]$ on a neighborhood of $0\in\mathbb{R}^{2n+1}$. 
To see this, it is enough to show that
for each $i\in\{1,\ldots,2n+1\}$, $j\in\{1,\ldots,2n\}$, 
$a\ge1$, $j_1,\ldots,j_a\in\{0,\ldots,2n\}$ and $\ep'>0$, 
\begin{align}
&\sup_{\ep\in(0,1]}\sup_{\lvert u\rvert\le\ep'}
\lvert a^{i,\ep}_j(u)\rvert<\infty, 
\qquad 
\sup_{\ep\in(0,1]}\sup_{\lvert u\rvert\le\ep'}
\Bigl\lvert\frac{\pa}{\pa u^J}a^{i,\ep}_j(u)\Bigr\rvert<\infty, \label{eq:aij bdd}\\
&\sup_{\ep\in(0,1]}\sup_{\lvert u\rvert\le\ep'}
\lvert b^{i,\ep}(u)\rvert<\infty, 
\qquad 
\sup_{\ep\in(0,1]}\sup_{\lvert u\rvert\le\ep'}
\Bigl\lvert\frac{\pa}{\pa u^J}b^{i,\ep}(u)\Bigr\rvert<\infty\label{eq:bi bdd}
\end{align}
hold, where $J=(j_1,\ldots,j_a)$ and $\frac{\pa}{\pa u^J}=\frac{\pa}{\pa u^{j_1}}\cdots\frac{\pa}{\pa u^{j_b}}$. 

We show \eqref{eq:aij bdd} for $i=2n+1$ and proofs of other cases are similar or more clear.
We have by definition of $a^{2n+1,\ep}$ and \eqref{eq:aij} that 
\[
a^{2n+1,\ep}_{j}(u)=\ep^{-1}a^{2n+1}_j(\lambda_{\ep}(u))
=\sum_{k=1}^{2n}u^k\frac{\pa a^{2n+1}_j}{\pa u^k}(\lambda_{c\ep}(u))
+2c\ep u^{2n+1}\frac{\pa a^{2n+1}_j}{\pa u^{2n+1}}(\lambda_{c\ep}(u))
\]
for some $c\in(0,1)$ depending on $\ep$ and $u$. 
This shows the first formula in \eqref{eq:aij bdd}. 
We also have
\[
\frac{a^{2n+1,\ep}_j}{\pa u^J}(u)
=\ep^{\lVert J\rVert-1}\frac{\pa a^{2n+1}_j}{\pa u^J}(\lambda_{\ep}(u)),
\]
which shows the second formula in \eqref{eq:aij bdd}.

Therefore we apply \cite[Theorem 2.17]{kusuoka-stroock} to
$\{\XH^{\ep}_j\}_{j=0,\ldots,2n}$, together with 
Lemma \ref{lem:Hormander}, we have
\begin{align*}
\mathbb{P}(\det\sigma^{\ep}\le K^{-\frac{1}{3(2n+1)}})
&\le
\mathbb{P}(\inf_{\xi\in S^{2n}}\langle\xi,\sigma^{\ep}\xi\rangle_{\mathbb{R}^{2n+1}}
\le K^{-\frac{1}{3}})\\
&\le \mu_1\exp(-\mu_2K^{\mu_3}),\quad K\in[1,\infty)
\end{align*}
with positive constants $\mu_1,\mu_2,\mu_3$ which are independent of $\ep$.
This implies the desired result.
\end{proof}

By Lemma \ref{lem:nondegenerate} and the Taylor expansion
\[
\delta_0(x+y)
=\delta_0(x)+
\sum_{a=1}^{\infty}\frac{\ep^a}{a!}
\sum_{\lvert J\rvert=a}\mathbb{E}
[\pa^J\delta_0(x)y^J],
\]
where $y^J=y^{j_1}\cdots y^{j_a}$ and 
$\pa^J=\frac{\pa^a}{\pa u^{j_1}\cdots\pa u^{j_a}}$ for $J=(j_1,\ldots,j_a)$, 
we apply \cite[Theorem V.9.4]{i-w} to have 
\begin{equation}\label{eq:delta_expansion}
\ep^{2n+2}\mathbb{E}[\delta_0(U^{\ep}_1)]
=\mathbb{E}[\delta_0(\mathbb{X}_1)]
+\sum_{a=1}^{\infty}\frac{\ep^a}{a!}
\sum_{\lvert J\rvert=a}\mathbb{E}
[\pa^J\delta_0(\mathbb{X}_1)\Psi^{1,\ep,J}_1].
\end{equation}

Since we have by definition
\begin{equation}
\Psi^{1,\ep}_1=\sum_{a=2}^{A+1}\ep^{a-2}\varphi^a_1
+\ep^A\Psi^{A+1,\ep}_1, \label{eq:Psi}
\end{equation}
where $\varphi^A_t=(\varphi^{A,1}_t,\ldots,\varphi^{A,2n+1}_t)$, 
we arrive at the expression as follows: 

\begin{proposition}\label{prop:d_A}
We have $\ep^{2n+1}\mathbb{E}[\delta_0(U^{\ep}_1)]=d_0+d_1\ep+d_2\ep^2+\cdots$ with
\begin{align}
d_0&=\mathbb{E}[\delta_0(\mathbb{X}_1)], \nonumber\\
d_A&=\sum_{a=1}^A\frac{1}{a!}
\sum_{(*)}
(\XH_{J_1}u^{i_1})(0)\cdots(\XH_{J_a}u^{i_a})(0)
\mathbb{E}[\pa^I\delta_0(\mathbb{X}_1)B^{J_1}\cdots B^{J_a}], \quad A\ge1, \label{eq:d_A}
\end{align}
where the summation $(*)$ is taken over the set
\begin{align*}
&\{(I,J_1,\ldots,J_a);I=(i_b)_{b=1,\ldots,a},1\le i_b\le 2n+1,
\lVert J_b\rVert\ge 2+\delta_{i_b,2n+1},
\\
&\qquad\qquad\qquad\qquad\lVert J_1\rVert+\cdots+\lVert J_a\rVert=A+\lVert I\rVert
\}.
\end{align*}
\end{proposition}
\begin{proof}
From \eqref{eq:delta_expansion} and \eqref{eq:Psi} we have the expansion
\[
\ep^{2n+1}\mathbb{E}[\delta_0(U^{\ep}_1)]=d_0+d_1\ep+d_2\ep^2+\cdots
\] 
with $d_0=\mathbb{E}[\delta_0(\mathbb{X}_1)]$ and 
\begin{align*}
d_A=\sum_{a=1}^A
\frac{1}{a!}\sum_{i_1,\ldots,i_a=1}^{2n+1}
\sum_{\substack{k_1,\ldots,k_a\ge0\\k_1+\cdots+k_a=A-a}}
\mathbb{E}[\pa^{(i_1,\ldots,i_a)}\delta_0(\mathbb{X}_1)
\varphi^{2+k_1,i_1}_1\cdots\varphi^{2+k_a,i_a}_1]
\end{align*}
for $A\ge1$. \eqref{eq:varphi} yields
\begin{align*}
&\mathbb{E}[\pa^{(i_1,\ldots,i_a)}\delta_0(\mathbb{X}_1)
\varphi^{2+k_1,i_1}_1\cdots\varphi^{2+k_a,i_a}_1]\\
&\qquad=\sum_{\lVert J_1\rVert=l_1}\cdots\sum_{\lVert J_a\rVert=l_a}
(\XH_{J_1}u^{i_1})(0)\cdots(\XH_{J_a}u^{i_a})(0)
\mathbb{E}[\pa^{(i_1,\ldots,i_a)}(\mathbb{X}_1)B^{J_1}_1\cdots B^{J_a}_1]
\end{align*}
with $l_b=2+k_b+\delta_{i_b,2n+1}$. For fixed $I=(i_1,\ldots,i_a)$ we have 
\begin{align*}
&k_b\ge0,k_1+\cdots+k_a=A-a,\lVert J_b\rVert=2+k_b+\delta_{i_b,2n+1}\\
&\qquad\iff \lVert J_b\rVert\ge2+\delta_{i_b,2n+1},
\lVert J_1\rVert+\cdots+\lVert J_a\rVert=A+\lVert I\rVert,
\lVert J_b\rVert=2+k_b+\delta_{i_b,2n+1}, 
\end{align*}
the proposition follows. 
\end{proof}

If we set $\mathbb{X}'_t=(-B^1_t,\ldots,-B^{2n}_t,\sum_{\alpha=1}^nS^{\alpha}_t)$,
It is easy to observe that, for each term appealing in \eqref{eq:d_A}, 
\begin{align*}
&\mathbb{E}[\pa^{i_1\cdots i_a}(\mathbb{X}'_1)(-B^{J_1}_1)\cdots(-B^{J_a}_1)]\\
&\qquad
=(-1)^{\lVert I\rVert}(-1)^{\lVert J_1\rVert+\cdots\lVert J_a\rVert}
\mathbb{E}[\pa^{i_1\cdots i_a}(\mathbb{X}_1)B^{J_1}_1\cdots B^{J_a}_1]\\
&\qquad
=(-1)^A\mathbb{E}[\pa^{i_1\cdots i_a}(\mathbb{X}_1)B^{J_1}_1\cdots B^{J_a}_1]. 
\end{align*}
By considering the Brownian motion $(-B^1_t,\ldots,-B^{2n}_t)$ instead of 
$(B^1_t,\ldots,B^{2n}_t)$, it follows that 
\[
d_A=0\quad\text{if $A$ is odd. }
\]
Using this, if we set $\ep^2=t$ and $c_A=d_{2A}$, 
we have the desired expansion in Theorem \ref{thm:main}. 
In \eqref{eq:d_A}, we observe that 
$(\XH_{J_b}u^{i_b})(0)$ is not affected by $O^{\lVert J_b\rVert+1}$ terms, 
and if we note that 
\[
\lVert J_b\rVert \le A+\lVert I\rVert -2(a-1)\le A+2, 
\]
Theorem \ref{thm:F-S} shows that 
$d_A$ can be written as an at most $A$-th degree polynomial
of at most $A+2$-th derivatives of Christoffel symbols. 

It remains to calculate the leading term $c_0=d_0=\mathbb{E}[\delta_0(\mathbb{X}_1)]$. 
$\mathbb{E}[\delta_0(\mathbb{X}_1)]$ is equal to the value at 0 
of the density function of
$\mathbb{X}_1=(B^1_1,\ldots,B^{2n}_1,\sum_{\alpha=1}^nS^{\alpha}_1)$. 
By \cite[Th\'{e}or\`{e}m 1]{gaveau}, we see that 
\[
c_0=\frac{1}{(2\pi)^{n+1}}\int_{\mathbb{R}}\biggl(\frac{2\tau}{\sinh 2\tau}\biggr)^n d\tau
\]
and the proof of Theorem \ref{thm:main} is complete. 

\begin{remark}
By the integration-by-parts formula for generalized Wiener functionals
(cf, for example \cite{i-w}), we can write
\[
\mathbb{E}[\pa^I\delta_0(\mathbb{X}_1)B^{J_1}\cdots B^{J_a}]
=\mathbb{E}[\delta_0(\mathbb{X}_1)\Psi]
=\mathbb{E}[\Psi|\mathbb{X}_1=0]p_{\mathbb{X}_1}(0)
\]
for a function $\Psi$ of $\{B_t\}_{t\in[0,1]}$, where
$p_{\mathbb{X}_1}$ is the density function of $\mathbb{X}_1$. 
$\Psi$ is generally so complicated to write down but some examples of calculation
will be shown in Section \ref{sec:ex}.
\end{remark}

\section{Examples}\label{sec:ex}
We show two examples of calculating the coefficients in the asymptotic expansion 
in Theorem \ref{thm:main}. 
One is the Heisenberg group, which is a trivial example and all higher 
terms of the expansion vanish. 
The other is the CR sphere, in which some non-trivial higher terms appear. 

\subsection{The Heisenberg group}
Let $\mathbb{H}_n=\mathbb{C}^n\times\mathbb{R}$ be the $(2n+1)$-dimensional
Heisenberg group with a coordinate system
$(z,t)$, $z=(z^1,\ldots,z^n)\in\mathbb{C}^n$, $t\in\mathbb{R}$. 
$\mathbb{H}^n$ is a strictly pseudoconvex CR manifold whose CR structure is 
\[
T_{1,0}=\bigoplus_{\alpha=1}^n\mathbb{C}Z_{\alpha}, \quad
Z_{\alpha}=\frac{\pa}{\pa z^{\alpha}}+\kyosu\overline{z^{\alpha}}\frac{\pa}{\pa t}. 
\]
Note that $\{Z_{\alpha}\}_{\alpha=1,\ldots,n}$ is a global orthonormal frame 
for $T_{1,0}$. 
It is easy to see that the associated covariant derivation is a null mapping, i.e.
\[
\nabla_{AB}^C=0 \quad\text{for any $A,B,C\in\{0,1,\ldots,n,\overline{1},\ldots,\overline{n}\}$}. 
\]
The structure functions introduced in Section \ref{sec.Folland-Stein} are 
\[
C^{2n+1}_{\alpha,n+\alpha}=2,\qquad C^{2n+1}_{n+\alpha,\alpha}=-2, \quad
\alpha=1,\ldots,n
\]
and other $C^i_{jk}$'s are zero. Then we have
\[
\XH_{\alpha}=\frac{\pa}{\pa u^{\alpha}}-u^{n+\alpha}\frac{\pa}{\pa u^{2n+1}}, \qquad
\XH_{n+\alpha}=\frac{\pa}{\pa u^{n+\alpha}}+u^{\alpha}\frac{\pa}{\pa u^{2n+1}}, \quad
\alpha=1,\ldots,n,
\]
therefore we see that $\varphi^a_t=0$ for any $a\ge2$. 
It follows that for $C^{\mathbb{H}_n}_a=0$ for $a\ge1$. 

\begin{remark}
This result also follows immediately from \cite[Th\'{e}or\`{e}m 1]{gaveau}, which says
$p(t,x,x)=t^{-n-1}\frac{1}{(2\pi)^{n+1}}\int_{\mathbb{R}}(\frac{2\tau}{\sinh 2\tau})^n d\tau$. 
\end{remark}

\subsection{The CR sphere}
Let $S^{2n+1}=\{(z^1,\ldots,z^{n+1})\in\mathbb{C}^{n+1};
\sum_{\alpha=1}^{n+1}\lvert z^{\alpha}\rvert^2=1\}\subset\mathbb{C}^{n+1}$ be 
the $(2n+1)$-dimensional unit sphere which we regard as a submanifold of $\mathbb{C}^{n+1}$. 
$S^{2n+1}$ has a CR structure 
\begin{align*}
(T_{1,0})_z
&=(T^{1,0}\mathbb{C}^{n+1})_z\cap\mathbb{C}T_zS^{2n+1}\\
&=\biggl\{\sum_{\alpha=1}^{n+1}v^{\alpha}\frac{\pa}{\pa z^{\alpha}};
\sum_{\alpha=1}^nv^{\alpha}\overline{z^{\alpha}}=0\biggr\}, \quad
z=(z^1,\ldots,z^{n+1})\in S^{2n+1}. 
\end{align*}

We can take 
\begin{align*}
\theta&=\kyosu\sum_{\alpha=1}^{n+1}
(-\overline{z^{\alpha}}dz^{\alpha}+z^{\alpha}\overline{z^{\alpha}}), \\
T&=\frac{\kyosu}{2}\sum_{\alpha=1}^{n+1}
\Bigl(z^{\alpha}\frac{\pa}{\pa z^{\alpha}}
-\overline{z^{\alpha}}\frac{\pa}{\pa \overline{z^{\alpha}}}\Bigr)
\end{align*}
and then the Levi form is
\[
L_{\theta}=2\sum_{\alpha=1}^{n+1}dz^{\alpha}\wedge d\overline{z^{\alpha}}, 
\]
which is positive definite, thus $S^{2n+1}$ is strictly pseudoconvex. 

Set
\[
T_{\alpha}=\frac{\pa}{\pa z^{\alpha}}
-\overline{z^{\alpha}}\sum_{\beta=1}^{n+1}z^{\beta}\frac{\pa}{\pa z^{\beta}}
\in T_{1,0}
\]
and $T_{\overline{\alpha}}=\overline{T_{\alpha}}$, 
then we have
\[
L_{\theta}(T_{\alpha},T_{\overline{\beta}})
=\delta_{\alpha\beta}-\overline{z^{\alpha}}z^{\beta}. 
\]

In the open subset $U=\{z;z^{n+1}\ne0\}$, if we set 
\[
Z_{\alpha}=T_{\al}-
\frac{\ov{z^{\al}}z^{n+1}}{\lvert z^{n+1}\rvert (1+\lvert z^{n+1}\rvert)}T_{n+1}, \quad
\al=1,\ldots,n, 
\]
then $\{Z_{\al}\}_{\al=1,\ldots,n}$ is a local orthonormal frame for $T_{1,0}$ on $U$. 
A direct calculation shows that for $\alpha,\beta=1,\ldots,n$, 
\begin{align}
[Z_{\al},Z_{\be}]
&=\frac{1}{1+\lvert z^{n+1}\rvert}(-\ov{z^{\be}}Z_{\al}+\ov{z^{\al}}Z_{\be}), \nonumber\\
[Z_{\al},Z_{\ov{\be}}]
&=\Bigl(\frac{\delta_{\al\be}}{1+\lvert z^{n+1}\rvert}
+\frac{\ov{z^{\al}}z^{\be}}{2\lvert z^{n+1}\rvert(1+\lvert z^{n+1}\rvert)^2}\Bigr)
\sum_{\gamma=1}^n(z^{\gamma}Z_{\gamma}+\overline{z^{\gamma}}Z_{\ov{\gamma}})
-2\kyosu\delta_{\alpha\beta}T, \label{eq:CR_Z}\\
[T,Z_{\alpha}]&=-\frac{\kyosu}{2}Z_{\al}. \nonumber
\end{align}
By rewriting these equalities in terms of $\XH_i$, 
we have the structure functions:
\begin{align*}
C_{\al\be}^{\gamma}
&=\frac{1}{\sqrt{2}}\Bigl(
\frac{-\delta_{\alpha\gamma}x^{\beta}+\delta_{\be\gamma}x^{\al}}{1+\lvert z^{n+1}\rvert}
+\frac{(x^{\al}y^{\be}-x^{\be}y^{\al})y^{\gamma}}{\lvert z^{n+1}\rvert(1+\lvert z^{n+1}\rvert)^2}
\Bigr),\\
C_{\al\be}^{n+\gamma}
&=\frac{1}{\sqrt{2}}\Bigl(
\frac{-\delta_{\alpha\gamma}y^{\beta}+\delta_{\be\gamma}y^{\al}}{1+\lvert z^{n+1}\rvert}
+\frac{(x^{\al}y^{\be}-x^{\be}y^{\al})x^{\gamma}}{\lvert z^{n+1}\rvert(1+\lvert z^{n+1}\rvert)^2}
\Bigr),\\
C_{n+\al,n+\be}^{\gamma}
&=\frac{1}{\sqrt{2}}\Bigl(
\frac{\delta_{\alpha\gamma}x^{\beta}-\delta_{\be\gamma}x^{\al}}{1+\lvert z^{n+1}\rvert}
+\frac{(x^{\al}y^{\be}-x^{\be}y^{\al})y^{\gamma}}{\lvert z^{n+1}\rvert(1+\lvert z^{n+1}\rvert)^2}
\Bigr),\\
C_{n+\al,n+\be}^{n+\gamma}
&=\frac{1}{\sqrt{2}}\Bigl(
\frac{\delta_{\alpha\gamma}y^{\beta}-\delta_{\be\gamma}y^{\al}}{1+\lvert z^{n+1}\rvert}
+\frac{(x^{\al}y^{\be}-x^{\be}y^{\al})x^{\gamma}}{\lvert z^{n+1}\rvert(1+\lvert z^{n+1}\rvert)^2}
\Bigr),\\
C_{\al,n+\be}^{\gamma}
&=\frac{1}{\sqrt{2}}\Bigl(
\frac{\delta_{\alpha\gamma}y^{\beta}-\delta_{\be\gamma}y^{\al}+2\delta_{\al\be}y^{\gamma}}{1+\lvert z^{n+1}\rvert}
+\frac{(x^{\al}x^{\be}+y^{\al}y^{\be})y^{\gamma}}{\lvert z^{n+1}\rvert(1+\lvert z^{n+1}\rvert)^2}
\Bigr),\\
C_{\al,n+\be}^{n+\gamma}
&=\frac{1}{\sqrt{2}}\Bigl(
\frac{-\delta_{\alpha\gamma}x^{\beta}+\delta_{\be\gamma}x^{\al}+2\delta_{\al\be}x^{\gamma}}{1+\lvert z^{n+1}\rvert}
+\frac{(x^{\al}x^{\be}+y^{\al}y^{\be})x^{\gamma}}{\lvert z^{n+1}\rvert(1+\lvert z^{n+1}\rvert)^2}
\Bigr),
\end{align*}
\[
C_{\al\be}^{2n+1}=0,\qquad
C_{n+\al,n+\be}^{2n+1}=0,\qquad
C_{\al,n+\be}^{2n+1}=2\delta_{\al\be},
\]
\[
C_{2n+1,\al}^{\be}=0,\qquad
C_{2n+1,\al}^{n+\be}=\frac{\delta_{\al\be}}{2},\qquad
C_{2n+1,n+\al}^{\be}=-\frac{\delta_{\al\be}}{2},\qquad
C_{2n+1,n+\al}^{n+\be}=0,
\]
where $\alpha,\beta,\gamma\in\{1,\ldots,n\}$ and
$x^{\alpha}=\Re z^{\alpha}$, $y^{\alpha}=\Im z^{\alpha}$. 

We can also calculate Christoffel symbols, that is, 
by comparing \eqref{eq:CR_Z} and \eqref{eq:Z_bracket}, we have
\begin{align*}
\Gamma_{\ov{\be}\al}^{\gamma}
&=\Bigl(
\frac{\delta_{\al\be}}{1+\lvert z^{n+1}\rvert}
+\frac{\ov{z^{\al}}z^{\be}}{2\lvert z^{n+1}\rvert(1+\lvert z^{n+1}\rvert)^2}
\Bigr)z^{\gamma}, \\
\Gamma_{\al\ov{\be}}^{\ov{\gamma}}
&=\Bigl(
\frac{\delta_{\al\be}}{1+\lvert z^{n+1}\rvert}
+\frac{\ov{z^{\al}}z^{\be}}{2\lvert z^{n+1}\rvert(1+\lvert z^{n+1}\rvert)^2}
\Bigr)\ov{z^{\gamma}} 
\end{align*}
for $\alpha,\beta,\gamma\in\{1,\ldots,n\}$. In particular, we have
\[
\Gamma_{\ov{\be}\be}^{\al}=
\Bigl(
\frac{1}{1+\lvert z^{n+1}\rvert}
+\frac{\lvert z^{\beta}\rvert^2}{2\lvert z^{n+1}\rvert(1+\lvert z^{n+1}\rvert)^2}
\Bigr)z^{\alpha}, \quad
\alpha,\beta=1,\ldots,n.
\]

Now let $x=(0,\ldots,0,1)\in S^{2n+1}$ and 
we consider the Folland-Stein normal coordinate $\{u^j\}_{j=1,\ldots,2n+1}$ around $x$. 
Theorem \ref{thm:F-S} and the calculation of $C^i_{jk}$'s above show that 
\begin{align*}
\XH_{\alpha}
&=\frac{\pa}{\pa u^{\alpha}}-u^{n+\alpha}\frac{\pa}{\pa u^{2n+1}}
+\frac{1}{12}\sum_{j=1}^{2n}(u^j)^2\frac{\pa}{\pa u^{\al}}
+\frac{1}{4}u^{2n+1}\frac{\pa}{\pa u^{n+\al}}\\
&\qquad
-\sum_{j=1}^{2n}\Bigl(\frac{1}{12}u^{\al}u^j+O^3\Bigr)\frac{\pa}{\pa u^j}
+\Bigl(\frac{1}{12}u^{\al}u^{2n+1}-\frac{1}{8}u^{n+\al}\sum_{j=1}^{2n}(u^j)^2+O^4\Bigr)
\frac{\pa}{\pa u^{2n+1}}, \\
\XH_{n+\alpha}
&=\frac{\pa}{\pa u^{n+\alpha}}-u^{\alpha}\frac{\pa}{\pa u^{2n+1}}
-\frac{1}{4}u^{2n+1}\frac{\pa}{\pa u^{\al}}
+\frac{1}{12}\sum_{j=1}^{2n}(u^j)^2\frac{\pa}{\pa u^{n+\al}}\\
&\qquad
-\sum_{j=1}^{2n}\Bigl(\frac{1}{12}u^{n+\al}u^j+O^3\Bigr)\frac{\pa}{\pa u^j}
+\Bigl(\frac{1}{12}u^{n+\al}u^{2n+1}+\frac{1}{8}u^{\al}\sum_{j=1}^{2n}(u^j)^2+O^4\Bigr)
\frac{\pa}{\pa u^{2n+1}}. 
\end{align*}
We also note that $\Gamma_{\ov{\be}\be}^{\al}(0)=\Gamma_{\be\ov{\be}}^{\ov{\al}}(0)=0$. 
Then we have 
\[
\varphi^{2,i}_t=0,\qquad i=1,\ldots,2n+1. 
\]
From this, Theorem \ref{thm:main} and Proposition \ref{prop:d_A} show that 
\[
c^{S^{2n+1}}_1(x)
=\sum_{i=1}^{2n+1}\mathbb{E}[\pa^i\delta_0(\mathbb{X}_1)\varphi^{3,i}_1]. 
\]
We also have
\begin{align*}
\varphi^{3,\alpha}_t
&=\sum_{j=1}^{2n}\Bigl(\frac{1}{6}B^{(j,j,\alpha)}_t
-\frac{1}{12}B^{(j,\alpha,j)}_t-\frac{1}{12}B^{(\al,j,j)}_t\Bigr)\\
&\qquad\qquad\qquad+\sum_{\beta=1}^n\Bigl(
\frac{1}{4}B^{(n+\be,\be,n+\alpha)}_t-\frac{1}{4}B^{(\be,n+\be,n+\alpha)}_t
\Bigr)
-\frac{n}{2}B^{(\al,0)}_t, \\
\varphi^{3,n+\alpha}_t
&=\sum_{j=1}^{2n}\Bigl(\frac{1}{6}B^{(j,j,\alpha)}_t
-\frac{1}{12}B^{(j,\alpha,j)}_t-\frac{1}{12}B^{(\al,j,j)}_t\Bigr)\\
&\qquad\qquad\qquad+\sum_{\beta=1}^n\Bigl(
-\frac{1}{4}B^{(n+\be,\be,\alpha)}_t+\frac{1}{4}B^{(\be,n+\be,\alpha)}_t
\Bigr)
-\frac{n}{2}B^{(\al,0)}_t, \\
\varphi^{3,2n+1}_t
&=
\sum_{\alpha=1}^n\sum_{j=1}^{2n}\Bigl(
\frac{5}{12}B^{(j,j,\al,n+\al)}_t
-\frac{5}{12}B^{(j,j,n+\al,\al)}_t
+\frac{1}{6}B^{(j,\al,j,n+\al)}_t
-\frac{1}{6}B^{(j,n+\al,j,\al)}_t\\
&\qquad\qquad\qquad
-\frac{1}{12}B^{(n+\al,j,\al,j)}_t
+\frac{1}{12}B^{(\al,j,n+\al,j)}_t
-\frac{1}{12}B^{(j,n+\al,\al,j)}_t
+\frac{1}{12}B^{(j,\al,n+\al,j)}_t\Bigr)\\
&\qquad+
\sum_{\alpha=1}^n\Bigl(
-\frac{n}{2}B^{(\al,0,n+\al)}_t
+\frac{n}{2}B^{(n+\al,0,\al)}_t 
\Bigr). 
\end{align*}

Let us describe $c^{S^{2n+1}}_1(x)$ more specifically using 
modified derivative operator. 

Let $W^{2n}=\{w=(w^1,\ldots,w^{2n})\colon[0,1]\to\mathbb{R}^{2n};\text{continuous}, w(0)=0\}$ be
the $2n$-dimensional Wiener space and 
our Brownian motion is regarded as $B^i(w)=w^i$, $w\in W^{2n}$. 
Let $H$ be the Cameron-Martin subspace of $W^{2n}$ and 
$\nabla\colon\mathbb{D}^{-\infty}(\mathbb{R}^{2n})\to\mathbb{D}^{-\infty}(H\otimes\mathbb{R}^{2n})$, 
$\nabla^*\colon\mathbb{D}^{-\infty}(H\otimes\mathbb{R}^{2n})\to\mathbb{D}^{-\infty}(\mathbb{R}^{2n})$ be
the derivation of generalized Wiener functionals and its dual respectively. 
The inner product in $H$ is denoted by $\langle,\rangle$. 

Set $S_t=\sum_{\alpha=1}^n S^{\alpha}_t$, then
$\mathbb{X}_1=(B^1_1,\ldots,B^{2n}_1,S_1)$ yields 
$\nabla\mathbb{X}_1=(B^1_1,\ldots,B^{2n}_1,\nabla S_1)$, thus for $i=1,\ldots,2n$, 
\[
\langle \nabla\delta_0(\mathbb{X}_1),B^i_1\rangle
=\pa^i\delta_0(\mathbb{X}_1)+\pa^{2n+1}\delta_0(\mathbb{X}_1)\langle\nabla S_1,B^i_1\rangle. 
\]
Therefore we have
\[
\mathbb{E}[\pa^i\delta_0(\mathbb{X}_1)\Psi]
=\mathbb{E}[\delta_0(\mathbb{X}_1)\nabla^*(\Psi B^i_1)
-\pa^{2n+1}\delta_0(\mathbb{X}_1)\langle S_1,B^i_1\rangle\Psi]. 
\]
Define $\NH\colon\mathbb{D}^{-\infty}(\mathbb{R}^{2n})\to\mathbb{D}^{-\infty}(H\otimes\mathbb{R}^{2n})$ by
$\NH F=\nabla F-\sum_{i=1}^{2n}\langle \nabla F,B^i_1\rangle B^i_1$. 
Then we have $\NH(\delta_0(\mathbb{X}_1))=\pa^{2n+1}\delta_0(\mathbb{X}_1)\NH S_1$ and
\begin{align*}
\mathbb{E}[\pa^{2n+1}\delta_0(\mathbb{X}_1)\Psi]
&=\mathbb{E}[\langle\NH(\delta_0(\mathbb{X}_1)),\NH S_1\rangle\lVert\NH S_1\rVert^{-2}\Psi]\\
&=\mathbb{E}[\delta_0(\mathbb{X}_1)\nabla^*(\lVert\NH S_1\rVert^{-2}\NH S_1\Psi)]. 
\end{align*}
A direct calculation shows
\begin{align*}
\NH S^{\alpha}_1(t)&=-2\int_0^t B^{n+\alpha}_sds+2t\int_0^1B^{n+\alpha}_sds,\\
\NH S^{n+\alpha}_1(t)&=2\int_0^t B^{\alpha}_sds-2t\int_0^1B^{\alpha}_sds
\end{align*}
for $\alpha=1,\ldots,n$. 
Combining these and performing a direct but lengthy calculation, 
we can write the coefficient $c^{S^{2n+1}}_1(x)$ as follows:

\begin{proposition}
For $x=(0,\ldots,0,1)\in S^{2n+1}$, we have
$c^{S^{2n+1}}_1(x)=\mathbb{E}[\delta_0(\mathbb{X}_1)\Phi]$, where
\begin{align*}
\Phi&=
\frac{1}{6}\sum_{i,j=1}^{2n}\Phi_{jji}^i
-\frac{1}{12}\sum_{i,j=1}^{2n}\Phi_{jij}^i
-\frac{1}{12}\sum_{i,j=1}^{2n}\Phi_{ijj}^i\\
&\qquad
+\frac{1}{4}\sum_{\al,\be=1}^n\Phi_{n+\be,\be,n+\al}^{\al}
-\frac{1}{4}\sum_{\al,\be=1}^n\Phi_{\be,n+\be,n+\al}^{\al}
-\frac{1}{4}\sum_{\al,\be=1}^n\Phi_{n+\be,\be,\al}^{n+\al}
+\frac{1}{4}\sum_{\al,\be=1}^n\Phi_{\be,n+\be,\al}^{n+\al}\\
&\qquad
-\frac{n}{2}\sum_{i=1}^{2n}\Phi_{i,2n+1}^i
-\frac{n}{2}\sum_{\alpha=1}^{n}(\Phi_{\alpha,2n+1,n+\al}^{2n+1}-\Phi_{n+\al,2n+1,\al}^{2n+1})\\
&\qquad
+\frac{5}{12}\sum_{\al=1}^n\sum_{j=1}^{2n}(\Phi_{j,j,\al,n+\al}^{2n+1}-\Phi_{j,j,n+\al,\al}^{2n+1})
+\frac{1}{6}\sum_{\al=1}^n\sum_{j=1}^{2n}(\Phi_{j,\al,j,n+\al}^{2n+1}-\Phi_{j,n+\al,j,\al}^{2n+1})\\
&\qquad
+\frac{1}{6}\sum_{\al=1}^n\sum_{j=1}^{2n}(\Phi_{\al,j,j,n+\al}^{2n+1}-\Phi_{n+\al,j,j,\al}^{2n+1})
+\frac{1}{12}\sum_{\al=1}^n\sum_{j=1}^{2n}(\Phi_{\al,j,n+\al,j}^{2n+1}-\Phi_{n+\al,j,\al,j}^{2n+1})\\
&\qquad
+\frac{1}{12}\sum_{\al=1}^n\sum_{j=1}^{2n}(\Phi_{j,\al,n+\al,j}^{2n+1}-\Phi_{j,n+\al,\al,j}^{2n+1})
-\frac{1}{6}\sum_{\al=1}^n\sum_{j=1}^{2n}(\Phi_{\al,n+\al,j,j}^{2n+1}-\Phi_{n+\al,\al,j,j}^{2n+1}),
\end{align*}
\[
\mathbb{E}[\pa^i\delta_0(\mathbb{X}_1)B^J_1]=\mathbb{E}[\delta_0(\mathbb{X}_1)\Phi^i_J], \quad
i=1,\ldots,2n+1,J=(j_1,\ldots,j_a), j_b=0,\ldots,2n. 
\]
Moreover, for $i,j,k,l=1,\ldots,2n$ we have
\begin{align*}
\Phi^i_{jkl}
&=
-\delta_{ij}B^{(0,k,l)}_1-\delta_{ik}B^{(j,0,l)}_1-\delta_{il}B^{(j,k,0)}_1
+(\kappa^i_1\kappa_4+\kappa^i_2\kappa_3)B^{(j,k,l)}_1\\
&\qquad
-2\kappa^i_1\kappa_3\biggl(\sigma(j)B^{(\sigma(j)n+j,0,k,l)}_1
+\sigma(k)\int_0^1\int_0^s B^j_uB^{\sigma(k)n+k}_udu\circ dB^l_s\\
&\qquad\qquad\qquad
+\sigma(l)\int_0^1B^{(j,k)}_sB^{\sigma(l)n+l}_sds\biggr),\\
\Phi^i_{i,0}
&=
(\kappa^i_1\kappa_4+\kappa^i_2\kappa_3)B^{(i,0)}_1
-2\sigma(i)\kappa^i_1\kappa_3B^{(\sigma(i)n+i,0)}_1,\\
\Phi^{2n+1}_{ijkl}
&=-\kappa_4B^{(i,j,k,l)}_1\\
&\qquad
+2\kappa_3\biggl(
\sigma(i)B^{(\sigma(i)n+i,0,j,k,l)}_1
+\sigma(j)\int_0^1\int_0^s\int_0^u B^i_vB^{\sigma(j)n+j}_vdv\circ dB^k_v\circ dB^l_s\\
&\qquad\qquad\qquad
+\sigma(k)\int_0^1\int_0^sB^{(i,j)}_uB^{\sigma(k)n+k}_udu\circ dB^l_s
+\sigma(l)\int_0^1B^{(i,j,k)}_sB^{\sigma(l)n+l}_sds\biggr),\\
\Phi^{2n+1}_{i,0,j}
&=-\kappa_4B^{(i,0,j)}_1
+2\kappa_3\biggl(
\sigma(i)B^{(\sigma(i)n+i,0,0,j)}_1
+\sigma(j)\int_0^1B^{(i,0)}_sB^{\sigma(j)n+j}_sds\biggr), 
\end{align*}
where
\[
\sigma(i)=\begin{cases}1,&i=1,\ldots,n,\\-1,&i=n+1,\ldots,2n,\end{cases}
\]
\begin{align*}
\kappa^i_1&=-2\sigma(i)B^{(\sigma(i)n+i,0)}_1,\\
\kappa^i_2&=2B^{(i,0)}_1-4B^{(i,0,0)}_1,\\
\kappa_3&=\biggl(4\sum_{i=1}^{2n}\biggl(\int_0^1(B^i_t)^2dt-(B^{(i,0)}_1)^2\biggr)\biggr)^{-1},\\
\kappa_4&=-8(\kappa_3)^2
\sum_{\alpha=1}^n\biggl(
-2\int_0^1B^{(n+\al,0)}_tB^{\alpha}_tdt+2\int_0^1B^{(\al,0)}_tB^{n+\alpha}_tdt
\biggr)\\
&\qquad
+2\sum_{\alpha=1}^n(
B^{(\alpha,0)}_1B^{(n+\alpha,0,0)}_1-B^{(n+\alpha,0)}_1B^{(\alpha,0,0)}_1
). 
\end{align*}
\end{proposition}


\end{document}